\documentclass[11pt, reqno]{amsart}
\usepackage{amsthm}
\usepackage{enumerate}
\usepackage{indentfirst}
\usepackage{graphicx}
\usepackage[backref,colorlinks,linkcolor=red,citecolor=blue]{hyperref}
\usepackage{bm}

\theoremstyle{plain}
\newtheorem{theorem}{Theorem}

\newtheorem{corollary}{Corollary}

\theoremstyle{definition}

\theoremstyle{remark}

\title{Stability Preserving  Data-driven Models With Latent Dynamics}
\author{Yushuang Luo and  Xiantao Li  and Wenrui Hao}
\begin{document}

\begin{abstract}
 In this paper we introduce a data-driven modeling approach for dynamics problem with latent variables.  The state space of the proposed model includes artificial latent variables, in addition to observed variables that can be fit to a given data set. We present a model framework where the stability of the coupled dynamics can be easily enforced. The model is implemented by recurrent cells and trained using back propagation through time. Numerical examples using benchmark tests from order reduction problems demonstrate the stability of the model and the efficiency of the recurrent cell implementation. As applications, two fluid-structure interaction problems are considered to illustrate the accuracy and predictive capability of the model.
\end{abstract}

\maketitle

\section{Introduction}\label{sec: intro}

Forecasting the long-time behavior of  a complex system based on short-time data series 
 is a long-standing problem in many scientific domains, e.g., spacecraft designing \cite{juang1993identification} and meteorology \cite{khouider2010stochastic}. 
 One particular challenge is due to the fact that there are hidden (latent) dynamics that are not directly observed. More specifically, the dynamics of the observed quantities is the result of continuous interactions with the latent dynamics.  In addition, choosing an appropriate ODE model to fit is also crucial to the effectiveness of the method. One well established framework for constructing an effective model is reduced-order modeling (ROM), where one starts with an underlying full-order model (FOM), and derive an a reduced model, often by subspace projections \cite{freund1999reduced,Bai2002,antoulas2005approximation}. With certain guaranteed approximation properties, the reduced models are able to efficiently capture the input-output relation.    One can view this framework as a bottom-up approach in that the matrices in the FOM are accessible, and in this sense, they are intrusive. Furthermore, the ROMs can also be reformulated into a data-driven problem, e.g., by using the Loewner framework \cite{ionita2014data,karachalios2021loewner,antoulas2005approximation}, in conjunction with data in the frequency domain. Another alternative is the Proper Orthogonal Decomposition (POD), which makes use of the leading principal components computed from data \cite{berkooz1993proper, swischuk2019projection}. Overall, most methods for ROM are aimed to reproduce the input-output relation, where there is no feedback mechanism, i.e., the output has no influence on latent dynamics. 
 
 Meanwhile, many non-intrusive models have recently been developed. For example,  the dynamic mode decomposition (DMD) method \cite{schmid2010dynamic} uses linear regression to fit a linear ODE to the leading components of high dimensional time series data. The operator inference approach \cite{peherstorfer2016data} is an extension of DMD in the way that it adds polynomial matrix operators to the linear model. Sparse regressions are used in \cite{brunton2016discovering} to learn governing equations from a set of candidate terms. It has been shown that deep neural networks (DNN) are able to recover potentially non-autonomous dynamics from data \cite{churchill2022robust, qin2021data}. Another notable approach is the physics-informed neural networks (PINNs) \cite{raissi2019physics},  which may also be considered as non-intrusive since it assumes knowledge of physics constraints instead of the full underlying model. Finally, neural ODEs is another general methodology to represent and learn ODE models \cite{chen2018neuralode, rubanova2019latent}. Although the effectiveness of some of these methods are demonstrated without considering the influence of latent variables, in principle, one can append a linear latent dynamics, and used a regression method (e.g., the method in \cite{brunel2008parameter}) to determine the model parameters. 

 Motivated by the ROM technique, we set up ODE models that involve the dynamics of latent variables. But compared to standard ROM \cite{freund1999reduced,Bai2002,antoulas2005approximation,benner2015survey}, we make an important extension by introducing ``cross terms'', that model the interactions. Namely, the output from the latent dynamics contributes to the dynamics of the observed quantities, which in turn influence the latent dynamics. 

{The ability for an ODE model to predict long-term dynamics relies critically on the stability. Despite the aforementioned wide variety of ODE models that have been developed, identifying parameters that lead to stable solutions is under explored. Stability is often left alone or treated empirically \cite{vlachas2018data,pawar2020data,vlachas2020backpropagation}.  Although the stability, e.g., the linear stability, is easy to examine for a specific dynamical system, ensuring such stability in a parameter estimation algorithm is not a trivial task. In many cases, the stability condition can be traced to a stable matrix, i.e.,  the eigenvalues have non-positive real parts. Although there are algorithms to identify a nearest stable matrix  \cite{gillis2017computing, gillis2019approximating}, accessing the spectrum in a training algorithm is clearly not practical due to the added computational overhead. In ROM, the stability is often fixed after the model reduction procedure \cite{Bai2002, gosea2016stability}. But such fix can compromise the training accuracy. From a data driven perspective, unstable modes can not be suppressed by training on short time series data. Often observed is that the models can be fit with very small residual error, but when it is evaluated at later instances, the model becomes unstable, and it is unable to predict the long term dynamics.   
 
The stability in our proposed approach is maintained at two levels. First, the ODE model is constructed in a way that it comes naturally with a  Lyapunov functional.  This is an important departure from many existing frameworks. 
 Since our primary interest is the prediction of observed quantities, the specific interpretation of the latent variables is less relevant. In particular, we show that an orthogonal transformation of the latent dynamics does not change the parametric form of the model. A remarkable finding from this observation is that with a particular choice of the orthogonal transformation, the symmetric part of the stable matrix becomes diagonal, which offers a simple procedure to enforce stability. 
In practice, instability can emerge from either a time-continuous model or its discretization. Thus the second part of our approach is to continue to maintain stability at the discrete level, where we consider an implicit mid-point method discretization of the continuous model, and we prove that the discrete model inherits the stability the continuous model. 
 
 }

Part of our model is represented in a network structure. The discrete model is implemented by a recurrent cell, making it efficient to train thanks to back propagation through time (BPTT).
 But it is worthwhile to emphasize that the latent variables introduced in original work of  neural ODEs \cite{chen2018neuralode} represent a continuous-depth hidden layer in a neural network, which is different from the latent variables in the current work.

An important class of problems that exhibit the features of latent dynamics is fluid-structure interactions (FSIs), which arise in many applications in material science and biology \cite{peskin1972flow, peskin2002immersed, kleinstreuer2006biofluid, dowell2001modeling}.  But we also point out that such models are also representative of many other scenarios where continuous interactions exist between the observed quantities and the latent variables, e.g., protein dynamics in solvent \cite{schlick2010molecular}, heat conduction with generalized heat fluxes \cite{chu2019mori}.  
 In the context of fluid-structure interactions, observed data include time series of the structure, e.g., interfaces. Latent variables act as a reduction of the high-dimensional fluid variables  \cite{luo2022projection}. 
 A prescribed force exerted by the structure is coupled to the latent variables. Our models are strongly motivated by these problems. 

In general, there are many methods to estimate parameters in ODE models, see \cite{brunel2008parameter} and the references therein. Some of the algorithms can certainly be used to determine our model. It is worthwhile to point out that many parameter estimation algorithms have been developed for ergodic dynamical systems \cite{harlim2018data,iacus2008simulation,lin2021data,ma2019coarse,berry2020bridging}. The problems considered in this paper are deterministic and transient dynamics. 

The rest of the paper is organized as follows. In Section \ref{sec: model}, we first motivate the form of our model by fluid-structure interaction. Then we show the stability properties of our model, followed by an illustration of the recurrent cell implementation. In Section \ref{sec: examples}, we first compare our method to neural ODEs and discrete model without stability conditions on data from classic benchmark examples in model reduction. We then apply our method to two FSI examples. Our numerical results demonstrate the learned models’ stability and efficiency, as well as their predictive accuracy. Summary and future directions are given in Section \ref{sec: conclusions}.

\section{Stability preserving data-driven reduced model}\label{sec: model}
We use the notation $\lbrace (t_j,\tilde{\bm{x}}_j): \ 1 \leq j \leq N\rbrace$ to denote time series observed quantities $\tilde{\bm{x}}$ at  times instances $\lbrace t_1,\ldots,t_N\rbrace $ that are not necessarily evenly spaced. Here $N$ denotes the length of the time series. Our goal is to construct and learn an ROM including latent variables.
 Throughout the paper, we use lowercase for scalars, lowercase bold for vectors, and capital for matrices. $X^T$ stands for the transpose of a real matrix $X$, similarly for the transpose of a vector. By $\langle \cdot, \cdot \rangle$ and $\Vert \cdot \Vert$ we denote the inner product in Euclidean spaces and the norm induced by this inner product. We write $X\preceq 0$ ($X\succeq 0$) if $X$ is symmetric and negative semi-definite (symmetric positive semi-definite), respectively. 

\subsection{Stable continuous models}\label{sec: model_cts}
 Mimicking the general formulation of the immersed boundary method  with time-dependent stokes flow \cite{peskin2002immersed}, we propose a continuous model with latent variables, 
\begin{equation}\label{eq: fsi_full}\left\{
\begin{aligned}
\dot{\bm{u}}(t) = & W\bm{u}(t) + L \bm{f}(\bm{x}), \\
\dot{\bm{x}}(t) = & L^T\bm{u}(t),  
\end{aligned}\right.
\end{equation}
where $\bm{u}(t)\!: [0,+\infty) \rightarrow \mathbb{R}^p$ represents the latent variables, $\bm{x}(t)\!: [0,+\infty) \rightarrow \mathbb{R}^l$ are observed variables. $W\in\mathbb{R}^{p\times p}$ is a constant matrix. $L\in\mathbb{R}^{p\times l}$ is the fluid-structure coupling matrix. $\bm{f}(\bm{x})\!: \mathbb{R}^l \rightarrow \mathbb{R}^l$ can be viewed as the force exerted on the latent dynamics. The following condition guarantees the stability of a system in the form of \eqref{eq: fsi_full}.

\begin{theorem}\label{thm: fullfsi_stab}
Assume that $\bm{f}$ is a conservative force, i.e.,  $\bm{f}(\bm{x}(t)) = -\nabla_{\bm{x}}E\big(\bm{x}(t)\big)$ for some {positive} energy functional $E(\bm{x})\geq 0, ~ \forall \bm{x}$. The system \eqref{eq: fsi_full} is stable {in the sense of Lyapunov} if the symmetric part of $W$ is negative semi-definite, i.e., $W+W^T\preceq 0$.  
\end{theorem}
\begin{proof}
Consider the following Lyapunov functional 
\begin{equation}\label{eq: fom_lf}
V(\bm{u}(t), \bm{x}(t)) = \frac{1}{2}\Vert \bm{u}(t)\Vert^2 + E\big(\bm{x}(t)\big).
\end{equation}
$V\geq 0$ is trivial. The following calculation shows $\dot{V}:=\frac{d V}{dt} \leq 0$.
\begin{align*}
\dot{V} = & \frac{1}{2}\langle \dot{\bm{u}}(t), \bm{u}(t) \rangle + \frac{1}{2}\langle \bm{u}(t), \dot{\bm{u}}(t) \rangle +\langle \dot{\bm{x}},\nabla_{\bm{x}}E(\bm{x}(t)) \rangle\\
= & \frac{1}{2}\langle \bm{u}(t), (W+W^T)\bm{u}(t) \rangle+ \langle \bm{u}(t),L\bm{f}(\bm{x}(t))\rangle \\
&+ \langle -\bm{f}(\bm{x}(t)),L^T\bm{u}(t) \rangle \\
= &\frac{1}{2}\langle \bm{u}(t), (W+W^T)\bm{u}(t) \rangle \leq 0.
\end{align*}
\end{proof}

{In practice, it is difficult to enforce the stability $W+W^T\preceq 0$ for \eqref{eq: fsi_full}, especially when the dimension $p$ is large. To circumvent this issue, we consider a ``canonical form'', 
\begin{equation}\label{eq: data_linear}\left\{
\begin{aligned}
\dot{\bm{z}}(t) = & (D+S)\bm{z}(t) + R\bm{f}(\bm{x}(t)), \\
\dot{\bm{x}}(t) = & R^T\bm{z}(t), 
\end{aligned}\right.
\end{equation}
with diagonal $D$ and skew-symmetric $S$, which reduces the stability requirement to much fewer constraints. Remarkably this model has the same representability as  general models of the form  \eqref{eq: fsi_full}, where $\bm{z}$ is the "canonical" latent variable. From now on, we shall denote the latent variable by $\bm{z}$ instead of $\bm{u}$. }

\begin{theorem}\label{thm: equiv_ode}
{Let $\bm{u}(t)$, $\bm{x}(t)$ be a solution of any system in the form of \eqref{eq: fsi_full}. There exists a system in the form of \eqref{eq: data_linear}, with solutions given by $\bm{z}(t)\!:=P^T\bm{u}(t)$, and $\bm{x}(t)$ for some orthogonal matrix $P$.}
\end{theorem}
\begin{proof}
We decompose the square matrix $W$ to the sum of a symmetric matrix and a skew-symmetric matrix $W =  +(W-W^T)/2$.
Denote the diagonalization of the symmetric part by $(W+W^T)/2=PDP^T,$
where $D$ is diagonal and $P$ is orthogonal. Define $\bm{z}(t)\!:=P^T\bm{u}(t)$ and $R\!:=P^TL$. Then $\bm{z}(t)$ satisfies the following ODE
\begin{equation}
\begin{aligned}
\dot{\bm{z}}(t) = & P^T\dot{\bm{u}}(t) = P^TW\bm{u}(t) + P^TL \bm{f}(\bm{x}(t))  \\
= & P^T\left(\frac{W+W^T}{2}+\frac{W-W^T}{2} \right)P\bm{z}(t) + P^TL \bm{f}(\bm{x}(t))  \\
= & \left(P^T\frac{W+W^T}{2}P+P^T\frac{W-W^T}{2}P \right)\bm{z}(t) + P^TL \bm{f}(\bm{x}(t))  \\
= & (D+S)\bm{z}(t) +  R \bm{f}(\bm{x}(t)),
\end{aligned}
\end{equation}
where $S\!:=P^T\frac{W-W^T}{2}P$ is a skew-symmetric matrix. And it follows that 
\[
\dot{\bm{x}}(t) =  L^T\bm{u}(t) = L^TP\bm{z}(t) = R^T\bm{z}(t).
\]
\end{proof}

Applying theorem \ref{thm: fullfsi_stab} to the canonical form \eqref{eq: data_linear} yields the following stability condition.
\begin{corollary}
The ODE system \eqref{eq: data_linear} is stable if $\bm{f}$ is conservative and $D\preceq 0$.
\end{corollary}

Therefore, a system of \eqref{eq: data_linear} with constraint $D\preceq 0$ is equivalent to \eqref{eq: fsi_full} with constraint $W+W^T\preceq 0$. Next, we introduce the discrete model based on \eqref{eq: data_linear}.

\subsection{Stable discrete models}\label{sec: model_dis}

We now turn to the discrete model based on a discretization of \eqref{eq: data_linear}. {Since the stability properties are our primary focus, a straightforward discretization, e.g., the Euler's method and general explicit Runge-Kutta methods,  may not be a good choice because they do not automatically} inherit the stability of \eqref{eq: data_linear} when $D\preceq 0$. On the other hand, the implicit Euler's method is stable but may be more diffusive than the continuous system, limiting its capability to capturing certain properties in the data, e.g., periodicity. Therefore, the discrete model we propose is derived from the implicit mid-point scheme. The next theorem shows the implicit mid-point scheme inherits the Lyapunov stability \cite{deuflhard2002scientific}.
\begin{theorem}\label{thm: mid_quad}
Suppose that the solution of the autonomous ODE
\begin{equation}\label{eq: ode_quad}
    \dot{\bm{x}}=\bm{h}(\bm{x}),
\end{equation}
has stable solutions  with respect to a quadratic Lyapunov functional
\begin{equation}\label{eq: ode_lyap}
V(\bm{x}) = \frac{1}{2}\Bigl< \bm{x}, G\bm{x}\Bigr>,
\end{equation}
where $G \succeq 0$. Namely, $V(\bm{x})\geq 0$ and 
\begin{equation}\label{eq: ode_lv}
    \dot{V}\!:= \Bigl< \bm{h},  \nabla V(\bm{x}) \Bigr> \leq 0.
\end{equation}
  Then the implicit mid-point discretization
\begin{equation}\label{eq: ode_imp}
    \frac{\bm{x}_{i+1}-\bm{x}_i}{\Delta t} = \bm{h}\left(\frac{\bm{x}_i+\bm{x}_{i+1}}{2} \right), \; i\geq 0, 
\end{equation}
is unconditionally stable with respect to the same Lyapunov functional \eqref{eq: ode_lyap}.
\end{theorem}
\begin{proof}

Denote $\bm{x}_{i+1/2}:=\frac{\bm{x}_i+\bm{x}_{i+1}}{2}$, direct calculation shows $V(\bm{x}_{i+1}) \leq V(\bm{x}_i)$:
\begin{align*}
    V(\bm{x}_{i+1}) - V(\bm{x}_i)= & 
     \frac12 \Bigl< \bm{x}_{i+1}+\bm{x}_i,  G(\bm{x}_{i+1}-\bm{x}_i) \Bigr> \\ 
    = & \Delta t \Bigl< \bm{x}_{i+1/2}, G\bm{h}(\bm{x}_{i+1/2}) \Bigr> \leq  0.  
\end{align*}
\end{proof}

{The method \eqref{eq: ode_imp} is an implicit ODE method, which can be recast into the standard form \cite{deuflhard2002scientific},
\[ \bm{k} = \bm{h}\left(\bm{x}_i + \frac{\Delta t}2  \bm{k}\right), \quad  
\bm{x}_{i+1} = \bm{x}_i + \Delta t \bm{k}.\] 
Such a stability-preserving descretizations can be generalized to higher order using a collocation approach \cite{deuflhard2002scientific}.
}

We now return to the model \eqref{eq: data_linear}. By applying the 
 implicit mid-point discretization \eqref{eq: ode_imp}, one arrives at,
 \begin{equation}\label{eq: linear_mid_point}
\begin{aligned}
\bm{z}_{i+1} = & \bm{z}_i+\frac{\Delta t}{2}(D+S)(\bm{z}_i+\bm{z}_{i+1}) + \Delta tR\bm{f}\left(\frac{\bm{x}_{i+1}+\bm{x}_i}{2}\right), \\
\bm{x}_{i+1} = & \bm{x}_i + \frac{\Delta t}{2}R^T(\bm{z}_i+\bm{z}_{i+1}).
\end{aligned}
\end{equation}

 {When applying a gradient-based algorithm, e.g., the ADAM method \cite{kingma2014adam}, to learn model parameters of \eqref{eq: linear_mid_point} from data $\lbrace (t_j,\tilde{\bm{x}}_j): \ 1 \leq j \leq N\rbrace$, derivatives of $\bm{x}_i$ with respect to model parameters must be calculated for all $i$. Note that all $\bm{x}_i$'s in \eqref{eq: linear_mid_point} are model outputs and their distance to data $\tilde{\bm{x}}_j$'s will be minimized. For general $\bm{f}$,  such differentiation can be carried out by implicit differentiation, which involves solving a nonlinear equations for every $i$ during back propagation. 
 To simplify the implementation, we consider the following alternative of \eqref{eq: linear_mid_point} by linearizing $\bm{f}\left(\frac{\bm{x}_{i+1}+\bm{x}_i}{2}\right)$ about $\bm{x}_i$: }
 \begin{equation}\label{eq: linear_mid_discrete}
\begin{aligned}
\bm{z}_{i+1} = & \bm{z}_i+\frac{1}{2}(D_d+S_d)(\bm{z}_i+\bm{z}_{i+1}) +\frac{R_d}{2}J_f(\bm{x}_i)(\bm{x}_{i+1}-\bm{x}_i) \\
& +R_d\bm{f}(\bm{x}_i),  \\
\bm{x}_{i+1} = & \bm{x}_i + \frac{R_d^T}{2}(\bm{z}_i+\bm{z}_{i+1}), 
\end{aligned}
\end{equation}
where $J_f(\bm{x}_i)\!:=\nabla \bm{f}(\bm{x}_i)$ is the Jacobian of $\bm{f}$ (Hessian of $E$). The matrices $D_d$, $S_d$ and $R_d$ constitute the parameters that will be learned from data, with stability constraints $D_d\preceq 0$ and $S_d=-S_d^T$.

The loss of accuracy caused by the linearization may be mitigated by choosing small $\Delta t$. More importantly, we shall demonstrate that  the stability properties are preserved by such linearization.
Toward this end, we first examine the case when $\bm{f}$ is a linear function. Then \eqref{eq: linear_mid_discrete} is equivalent to a model in the form of \eqref{eq: linear_mid_point}. Thus, its stability property is not affected by the linearization.
\begin{theorem}\label{thm: gas_lf}
Model \eqref{eq: linear_mid_discrete} is globally stable if the $\bm{f}(\bm{x})$ is a conservative force of a quadratic energy $E(\bm{x})$, i.e., $\bm{f}(\bm{x})=-\nabla E(\bm{x})=-T\bm{x}$, where $T\succeq 0$. 
\end{theorem}
\begin{proof}
Model \eqref{eq: linear_mid_discrete} is the implicit mid-point discretization of a stable system with Lyapunov functional
\begin{equation}
V(\bm{z},\bm{x})=\frac{1}{2}\Vert \bm{z}\Vert^2+E(\bm{x}).
\end{equation}
$\bm{f}(\bm{x})=-\nabla E(\bm{x})=-T\bm{x}$ implies $E(\bm{x})=\frac{1}{2}\bm{x}^TT\bm{x}$ being quadratic. Thus, $V(\bm{z},\bm{x})$ is quadratic. Then by Theorem \ref{thm: mid_quad}, model \eqref{eq: linear_mid_discrete} is globally stable.
\end{proof}

If $\bm{f}$ is nonlinear, {one can study the linear stability around a mechanical equilibrium, i.e., $\bm{x}=\bm{x}_e$ with $\bm{f}(\bm{x}_e)=0$. A stability equilibrium corresponds to a positive semi-definite Hessian $J_f = \mathrm{Hess}(E)$.  We can show that  }
\eqref{eq: linear_mid_discrete} still inherits the linear stability, as stated by the following theorem.

\begin{theorem}\label{thm: ls_lf}
Model \eqref{eq: linear_mid_discrete} is locally stable about $\bm{x}=\bm{x}_e$ and $\bm{z}=0$ if $\bm{f}(\bm{x}_e)=0$ and $J_f(\bm{x}_e) \preceq 0$.
\end{theorem}
\begin{proof}
Consider the linearization of system \eqref{eq: data_linear} about $(\bm{z},\bm{x})=(\bm{0}, \bm{x}_e)$. If $\bm{f}(\bm{x}_e)=0$ and $J_f(\bm{x}_e) \preceq 0$, the linearized system is stable with a quadratic Lyapunov functional. Therefore, model \eqref{eq: linear_mid_discrete} is locally stable as it is the implicit mid-point discretization of the linearized system.
\end{proof}

By writing \eqref{eq: linear_mid_discrete} in a matrix form
\begin{equation}\label{eq: linear_mid_discrete_lf}
\begin{aligned}
&\begin{bmatrix}
I-\frac{1}{2}(D_d+S_d) & -\frac{1}{2}R_dJ_f(\bm{x}_i)\\-\frac{1}{2}R_d^T & I
\end{bmatrix} 
\begin{bmatrix}
\bm{z}_{i+1}\\ \bm{x}_{i+1}
\end{bmatrix}  \\
 =
&\begin{bmatrix}
I+\frac{1}{2}(D_d+S_d) & -\frac{1}{2}R_dJ_f(\bm{x}_i)\\\frac{1}{2}R_d^T & I
\end{bmatrix}
\begin{bmatrix}
\bm{z}_i\\ \bm{x}_i
\end{bmatrix}
+\begin{bmatrix}
R_d\bm{f}(\bm{x}_i)\\ 0
\end{bmatrix},
\end{aligned}
\end{equation}
one can see that the discrete model requires taking the inverse of a non-constant matrix depending on $\bm{x}_i$. Next, we establish the convexity condition  for such inverse to exist. 

\begin{theorem}\label{thm: inv_discrete}
The matrix on the left side of \eqref{eq: linear_mid_discrete_lf} is invertible {if and only if}
\begin{equation}\label{eq: invM}
M\!:= I-\frac{1}{2}(D_d+S_d)-\frac{1}{4}R_dJ_f(\bm{x}_i)R_d^T
\end{equation}
is invertible. {In particular, $M$ is invertible when $D_d\preceq 0$, $S_d=-S_d^T$, and $\langle\bm{u}, J_f(\bm{x}_i)\bm{u}\rangle \leq 0$ for any vector $\bm{u}$.}
\end{theorem}
\begin{proof}
Assuming $M$ is invertible, the formula for $2\times 2$ block matrices gives the inverse of the matrix on the left side of \eqref{eq: linear_mid_discrete_lf}
\begin{equation}
\begin{bmatrix}
M^{-1} & \frac{1}{2}M^{-1}R_dJ_f(\bm{x}_i) \\\frac{1}{2}R_d^TM^{-1} & I+\frac{1}{4}R_d^TM^{-1}R_dJ_f(\bm{x}_i)
\end{bmatrix}.
\end{equation}
Conversely, if $M$ is singular, there exists a nonzero vector $\tilde{\bm{z}}$ such that $M\tilde{\bm{z}}=\bm{0}$. Then the matrix on the left side of \eqref{eq: linear_mid_discrete_lf} is singular because
\[
\begin{bmatrix}
I-\frac{1}{2}(D_d+S_d) & -\frac{1}{2}R_dJ_f(\bm{x}_i)\\-\frac{1}{2}R_d^T & I
\end{bmatrix} 
\begin{bmatrix}
\tilde{\bm{z}}\\
\frac12 R_d^T\tilde{\bm{z}}
\end{bmatrix} = \bm{0},
\]
where the vector in the above equation is nonzero.

For $M$ to be invertible, it suffices to show that for any nonzero vector $\bm{u}$, $\langle\bm{u}, M\bm{u}\rangle > 0$. We decompose $M$ into $M=H-K$ with 
\[
H = I-\frac{1}{2}(D_d+S_d), \quad K = \frac{1}{4}R_dJ_f(\bm{x}_i)R_d^T.
\]
For any nonzero vector $\bm{u}$, we have
\begin{equation}\label{eq: invH}
\begin{aligned}
\langle\bm{u}, H\bm{u}\rangle = & \bm{u}^T(I-\frac{1}{2}(D_d+S_d))\bm{u} \\
= & \Vert \bm{u}\Vert^2 -\frac{1}{2}\bm{u}^TD_d\bm{u} - \frac{1}{2}\bm{u}^TS_d\bm{u} \\
= & \Vert \bm{u}\Vert^2 -\frac{1}{2}\bm{u}^TD_d\bm{u} \geq  \Vert \bm{u}\Vert^2 > 0.
\end{aligned}
\end{equation}
It follows that for any vector $\bm{u}$, $\langle\bm{u}, M\bm{u}\rangle = \langle\bm{u}, (H-K)\bm{u}\rangle >0$ since $\langle\bm{u}, H\bm{u}\rangle > 0$ (from \eqref{eq: invH}) and $\langle\bm{u},K\bm{u}\rangle \leq 0$ (due to the assumption $\langle\bm{u}, J_f(\bm{x}_i)\bm{u}\rangle \leq 0$ for any vector $\bm{u}$).
Therefore, $M$ is invertible.
\end{proof}

With the inverse given in theorem \ref{thm: inv_discrete}, the discrete model \eqref{eq: linear_mid_discrete_lf} is well-defined and could be written as
\begin{equation}\label{eq: linear_discrete_invd}
\begin{bmatrix}
\bm{z}_{i+1}\\ \bm{x}_{i+1}
\end{bmatrix}
 =
\begin{bmatrix}
2M^{-1}-I & 0\\R_d^TM^{-1} & I
\end{bmatrix}
\begin{bmatrix}
\bm{z}_i\\ \bm{x}_i
\end{bmatrix}
+\begin{bmatrix}
I\\ \frac{1}{2}R_d^T
\end{bmatrix}M^{-1}R_d\bm{f}(\bm{x}_i).
\end{equation}

\subsection{ Numerical implementation using recurrent cells}\label{sec: rnn}
We implement model \eqref{eq: linear_mid_discrete} as a customized recurrent cell. The dimension $p$ of the latent state $\bm{z}\in\mathbb{R}^p$ is a hyper parameter to be tuned. The function $\bm{f}$ and its Jacobian $J_f$ must also be given. The recurrent cell contains trainable parameters 
\begin{enumerate}[(1)]
\setlength\itemsep{-0.2em}
\item a vector in $\mathbb{R}^p$ of non positive values that encode the diagonal entries of $D_d$;
\item a dense matrix $C_d\in \mathbb{R}^{p\times p}$ that generates the skew-symmetric matrix $S_d=C_d-C_d^T$;
\item a dense matrix $R_d \in \mathbb{R}^{p\times l}$ that facilitates the communication with the latent variables.
\end{enumerate}
The recurrent cell is state-to-sequence in the sense that it takes the initial value $\bm{x}_0$ as input and outputs a sequence $\bm{x}_i\in \mathbb{R}^l$, $i=1,2,...,n$ of length $n$ (see Fig.~\ref{fig: rnn_cell} for a schematic of the recurrent architecture). 
\begin{figure}
    \centering
    \includegraphics[width=\columnwidth]{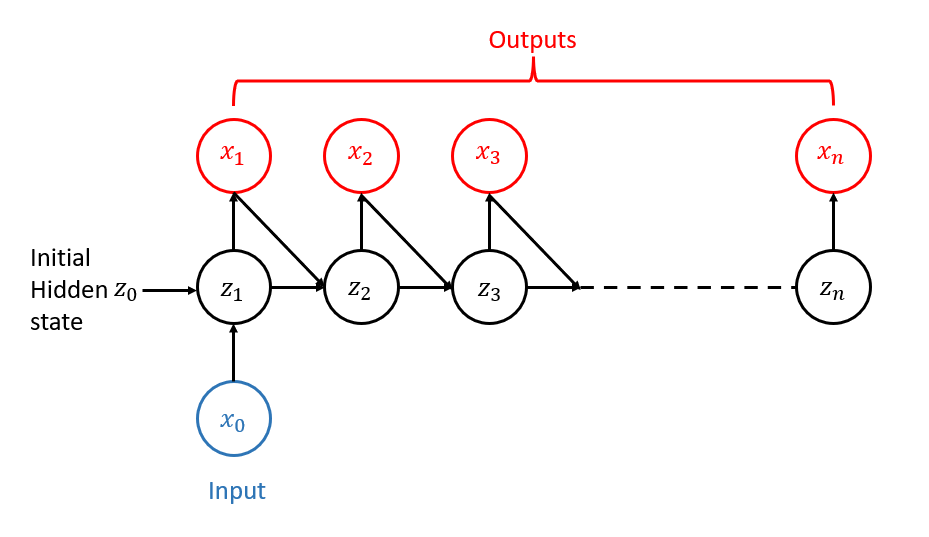}
    \caption{A state-to-sequence recurrent cell.}
    \label{fig: rnn_cell}
\end{figure} 

Note that \eqref{eq: linear_discrete_invd} is equivalent to
\begin{subequations}\label{eq: state_up}
\begin{alignat}{2}
      \bm{z}_{i+1} = & M^{-1}(2\bm{z}_i+R_d\bm{f}(\bm{x}_i))-\bm{z}_i  \label{eq: hid_up} \\
      \bm{x}_{i+1} = & \bm{x}_i+\frac{R_d^T}{2}(\bm{z}_i+\bm{z}_{i+1}). \label{eq: out_up}
    \end{alignat}
\end{subequations}
{During the forward propagation, \eqref{eq: state_up} is used instead of \eqref{eq: linear_discrete_invd} to simplify the update of $\bm{x}$. Specifically, it is done in three steps:}
\begin{itemize}
\setlength\itemsep{-0.2em}
    \item first, we calculate the inverse of $M$;
\item secondly, the latent state $\bm{z}_{i+1}$ is updated by \eqref{eq: hid_up};
\item finally, the output $\bm{x}_{i+1}$ of the next time step is calculated as \eqref{eq: out_up}.
\end{itemize}
Parameters of the recurrent cell, $(D_d, \ C_d, \ R_d)$, are trained by solving the following constrained optimization problem with the  ADAM optimizer. 
\begin{equation}\label{eq: loss_fcn}
\begin{array}{lc}
\displaystyle \min_{D_d,C_d,R_d} & \frac{1}{N}\displaystyle\sum_{j=1}^N\Vert \tilde{\bm{x}}_j-\bm{x}_{t_j/\Delta t}(D_d,C_d,R_d)\Vert^2\\
\textrm{subject to} & D_d\preceq 0.\\
\end{array}
\end{equation}

\section{Numerical Examples}\label{sec: examples}
In this section, {we present results from several  numerical experiments, which is organized as follows: First, }we  use data generated by linear FOMs first to show the importance of model stability and highlight the robustness of our method. Next, we apply our methods to the data generated by two FSI problems. The dynamics of the structures in both FSI problems are stable. In the first FSI problem, the structure settles to a stationary equilibrium state. The second one has a periodic motion. In both cases, the ROMs are able to learn the stable long term behavior from data.

All data sets are generated using Matlab R2020b. Reduced-order models are implemented by RNNs in \textit{Tensorflow} and trained on a single GTX 1660 GPU.

\subsection{Linear FOMs}\label{sec: ex_l}
We start with linear FOMs adopted from the Eady example \cite{Chahlaoui2002mrbenchmark},  a benchmark problem in  model reduction,
\begin{align}\label{eq: eady}
\dot{\bm{z}}(t) = & A\bm{z}(t)+\bm{b}u(t), \notag \\
y(t) = & \bm{c}\bm{z}(t),
\end{align}
with stable $A$ and $\bm{b}=\bm{c}^T$. We use matrices from the Eady example to construct FOMs in the form of \eqref{eq: fsi_full}. The solution data of $\bm{x}(t)$ is used to train ROMs. 

\subsubsection{Importance of stability}\label{sec: ex_l_s}
The FOM \eqref{eq: eady} takes the form of \eqref{eq: fsi_full} with $W$ and $L$ corresponding to $A$ and $\bm{b}$ in \eqref{eq: eady}, respectively. In this case, $p=598$ and $l=1$. We choose $f(x(t))=-x(t)$. This is a linear system whose eigenvalues have negative real parts, thus stable.

The data set consists of $80$ time series of $\lbrace (t_j,\tilde{\bm{x}}_j): \ 0 \leq j \leq 150\rbrace$ with evenly sampled $t_j=j\Delta t$. Each time series data is the solution of the FOM with the initial conditions $\bm{z}(0)=\bm{0}$ and $x(0)$ randomly sampled in the interval $[-1, 1]$. The data set is split into training/validation/testing sets with percentages given  by $70\%$, $20\%$, $10\%$, respectively. Only the first $15$ steps of each time series in the training and validation set will be used.

Our first two numerical experiments aim at examining the stability.
First, we train a discrete model based on the forward Euler scheme instead of the proposed discrete model in \eqref{eq: linear_mid_discrete}. This discretization yields the following discrete model 
\begin{equation}\label{eq: exp_mat}
\begin{bmatrix}
\bm{z}_{i+1}\\ 
x_{i+1} 
\end{bmatrix}
= 
\begin{bmatrix}
I+D_d+S_d & -R_d \\
R_d^T & I
\end{bmatrix}
\begin{bmatrix}
\bm{z}_i\\   
x_i 
\end{bmatrix},
\end{equation}
which has diagonal $D_d\preceq 0$ and skew-symmetric $S_d$. This discrete model is conditionally stable. We train the discrete model in the first 15 steps. As shown in Fig.~\ref{fig: exp_unstable}, the residual error is very small, indicating a very good fit. But the model prediction quickly diverges from the exact data and blows up. In fact, the spectral radius of the matrix on the right side of \eqref{eq: exp_mat} is greater than $1$,  which implies that the discrete system is unstable.

\begin{figure}
    \centering
    \includegraphics[width=\columnwidth]{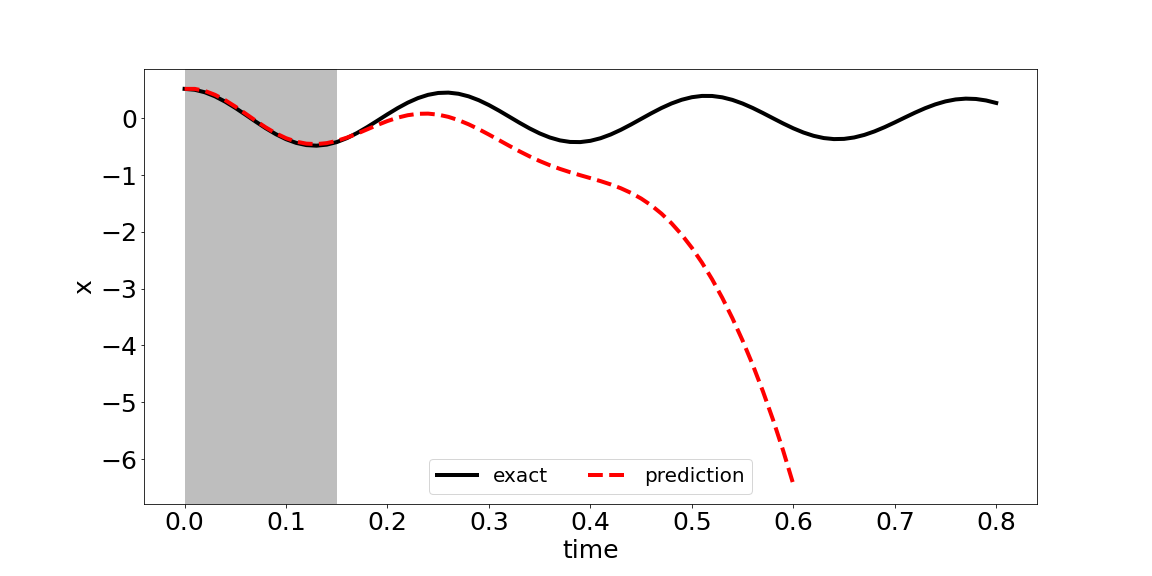}
    \caption{Example \ref{sec: ex_l_s}: A discrete model may be unstable even if it corresponds to a stable continuous model. The shaded area is the training data in the first 15 steps where the model performs well.}
    \label{fig: exp_unstable}
\end{figure}

Next, we train a discrete model in the form of \eqref{eq: linear_mid_discrete} but does not require $D_d\preceq 0$. Therefore, it is an implicit mid-point discretization of an ODE that is not necessarily stable. Therefore the stability of the discrete model itself is not guaranteed. Fig.~\ref{fig: mid_unstable} shows the comparison of the exact data and the model prediction on the same testing time series as in Fig.~\ref{fig: exp_unstable}. Compared to  \eqref{eq: exp_mat}, the model makes good predictions over a longer time period. However, the spectral radius of the iteration matrix of this discrete model is still larger than $1$, making it unstable.

\begin{figure}
    \includegraphics[width=\columnwidth]{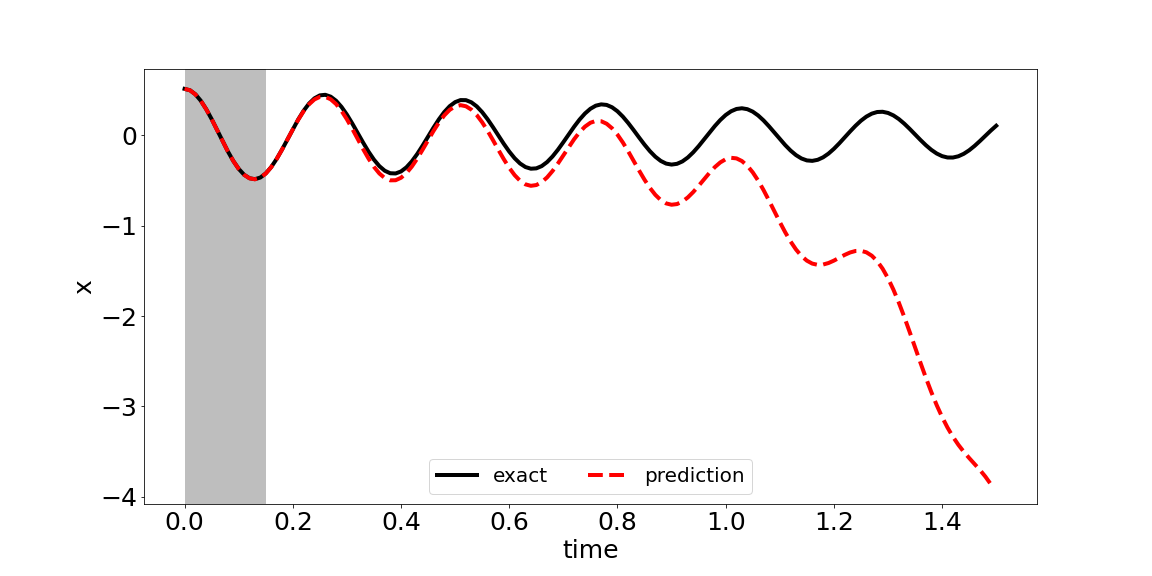}
    \caption{Example \ref{sec: ex_l_s}: Testing result of the model without $D_d\preceq 0$. A mid-point discretization may be unstable if it corresponds to an unstable continuous model. The shaded area covers the first 15 steps.}
    \label{fig: mid_unstable}
\end{figure}

On the other hand, our model \eqref{eq: linear_mid_discrete} with the constraint $D_d\preceq 0$ is stable since it satisfies theorem \ref{thm: gas_lf}. For the same testing time series as above, we observe that our model makes accurate long-term prediction and is theoretically guaranteed to converge to the same equilibrium state, 0, as the FOM.

\begin{figure}
    \centering
    \includegraphics[width=\columnwidth]{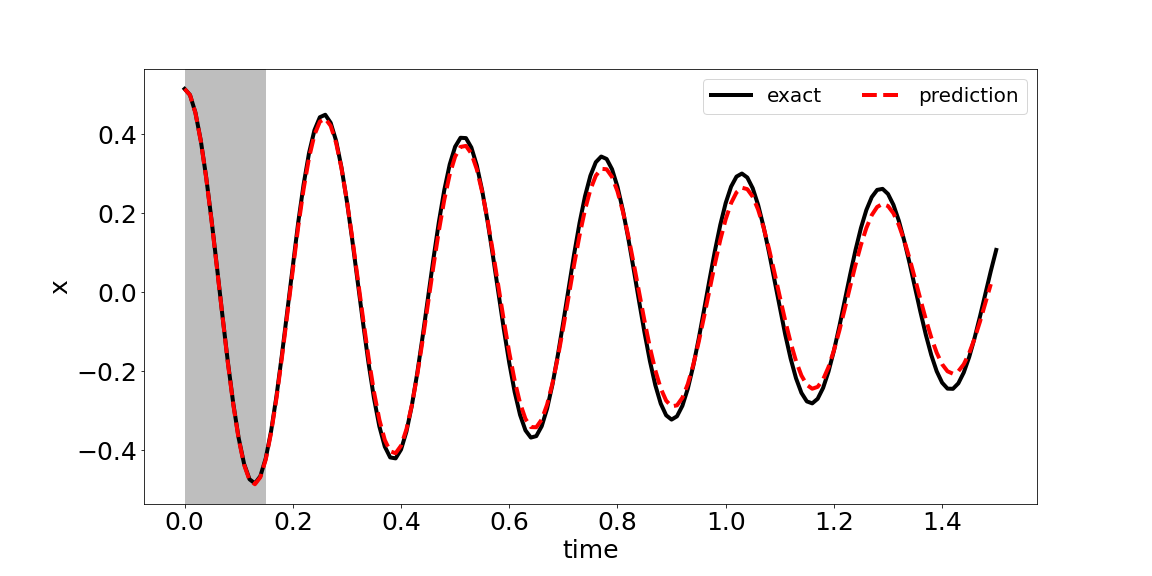}
    \caption{Example \ref{sec: ex_l_s}: Our model with stability condition makes accurate prediction on both the dynamics and the equilibrium state. The shaded area covers the first 15 steps.}
    \label{fig: imp_stable}
\end{figure}

\subsubsection{Robustness and efficiency of recurrent cell implementation}\label{sec: ex_l_m}
We now consider an FOM of the form \eqref{eq: fsi_full} with $l=16$. We still let $\bm{f}(\bm{x}(t))=-\bm{x}(t)$ similar to Section \ref{sec: ex_l_s}. $W$ is equal to matrix $A$ of \eqref{eq: eady}, but $L$ is a matrix rather than the vector $\bm{b}$  since $\bm{x}(t)$ is now a vector. We choose each column of $L$ by randomly shuffling entries of $\bm{b}$. It has been verified that all eigenvalues of matrix 
\begin{equation}
    \begin{bmatrix}
    A & -L\\
    L^T & 0
    \end{bmatrix}
\end{equation} 
have negative real parts so that the FOM is stable. The data set is generated in exactly the same way as in Section \ref{sec: ex_l_s}. 

Instead of training a discrete ROM, one could try directly training a continuous stable system \eqref{eq: data_linear} using neural ODE \cite{Chen_torchdiffeq_2021}. In this example, we compare the training efficiency and robustness of our method with neural ODE.  

Our approach yields a stable discrete ROM that learns from the first 15 steps and makes accurate long-term prediction (see Fig.~\ref{fig: mimo_proj}). It is also guaranteed that $\bm{x}$ vanishes as $t\to + \infty$.

\begin{figure}
    \centering
    \includegraphics[width=\columnwidth]{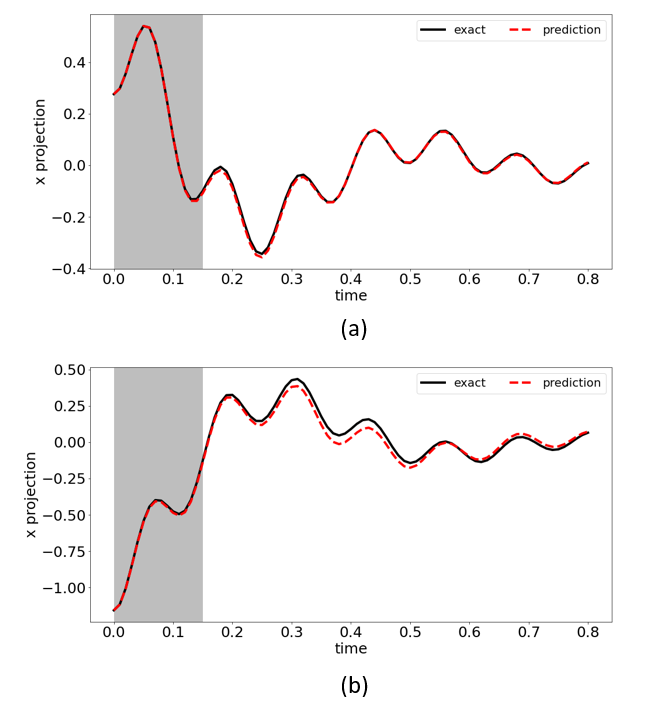}
    \caption{Example \ref{sec: ex_l_m}: Comparison between model prediction and the true data on testing time series. For visualization, the top and bottom figures plot scalar time series $\langle \bm{x}_i, \bm{v}_a\rangle$ and $\langle \bm{x}_i, \bm{v}_b\rangle$, respectively, where $\bm{v}_a$ and $\bm{v}_b$ are two distinct randomly sampled unit vectors. The shaded area covers the first 15 steps.}
    \label{fig: mimo_proj}
\end{figure} 

On the other hand, we could not successfully train a continuous model using neural ODE. The training always stops prematurely due to the ODE model \eqref{eq: data_linear} being stiff. This makes training more expensive, but less robust. In our experiments, it takes 0.5 seconds on average for an neural ODE to train one iteration, until an underflow error pops up. Our model trains much faster using back propagation through time (BPTT), taking 0.003 second per iteration on average. The training loss is recorded in Fig.~\ref{fig: mimo_logloss} from one experiment where we were able to manage to train 200 iterations with neural ODEs. We observe that when our model has already found a good discrete model, loss of the neural ODE has barely decreased. 

\begin{figure}
    \centering
    \includegraphics[width=\columnwidth]{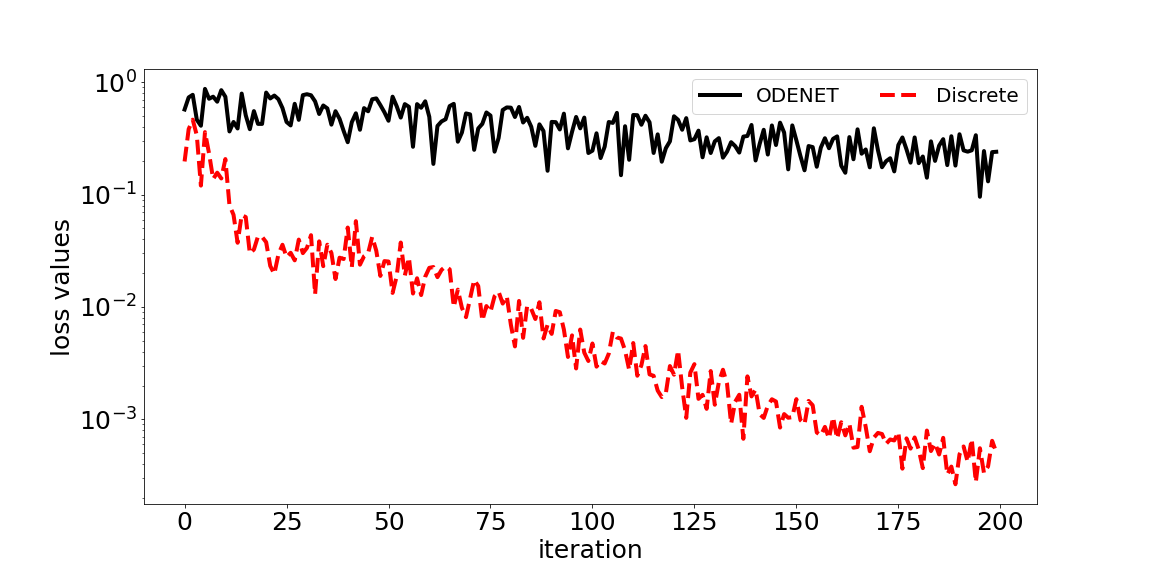}
    \caption{Example \ref{sec: ex_l_m}: The training loss of our model and a neural ODE. Weights are initialized by the Glorot initializer \cite{glorot2010understanding}. ADAM optimizer is used with learning rate $10^{-2}$.}
    \label{fig: mimo_logloss}
\end{figure}

\subsection{Rubber band oscillation}\label{sec: ex_rub}
The oscillation of a pressurized rubber band is a classic FSI example modeled by the immersed boundary method \cite{peskin2002immersed, battistaib2d}. The system consists of a stretched elastic rubber band, located in the center of a square fluid domain, with periodic boundary conditions on all edges. The resistance to stretching between successive structure points is achieved by modeling the connections with Hookean springs of resting length $R_L$ and spring stiffness $k_S$. If the virtual spring displacement is below or beyond $R_L$, the model will drive the system back towards a lower energy state. The elastic potential energy for a Hookean spring is given by
$ E_{spring}(\bm{x}) = \frac{1}{2}k_s\left(\Vert \bm{x}-\bm{x}_s\Vert - R_L\right)^2,$
where $\bm{x}_s$ represents the neighboring node coordinates of $\bm{x}$. The corresponding force is given by the gradient,
\begin{equation}\label{eq: f_spring}
    \bm{f}_{spring}(\bm{x})=-\nabla E_{spring}(\bm{x})=k_s\left(1-\frac{R_L}{\Vert \bm{x}-\bm{x}_s\Vert}\right)(\bm{x}_s-\bm{x}).
\end{equation}
Since the rubber band is stretched initially, it will start oscillating and eventually settle into an equilibrium state, as the fluid becomes stationary. In the case of $R_L=0$, \eqref{eq: f_spring} is linear. The zero resting length also means that the rubber band shrinks to a single point. The result of this case is shown in Section \ref{sec: ex_rub_l}. In Section \ref{sec: ex_rub_nl}, we train an ROM and test it for the case of nonzero $R_L$. This yields a circular equilibrium shape with nonzero radius for the rubber band, as well as nonlinear force function given by \eqref{eq: f_spring}.

For both cases, we sample 32 points on the structure, hence $l=64$. The spring stiffness is set to $k_s=2.5\times10^4$. Data is obtained by solving the fluid-structure interaction using a Matlab library \textit{IB2D} \cite{battistaib2d}. The data set contains one time series sampled with $\Delta t=10^{-3}$. The dimension of the latent state is set to $p=64$, same as the dimension of $\bm{x}$. We use the first 150 steps for training. The trained ROM is tested against the exact data from  later dynamics and the equilibrium state.

\subsubsection{Linear force ($R_L=0$)}\label{sec: ex_rub_l}
The training converges in less than 200 iterations (see Fig.~\ref{fig: rbl_loss}). To show the performance of the trained model, shapes of the rubber band at different time steps are compared (see Fig.~\ref{fig: rbl_snaps}). Small relative error of the rubber band perimeter (less than 3\%) is observed (see Fig.~\ref{fig: rbl_pm}). 

Due to the linearity of the force in this case, $J_f(\bm{x}_i)= T\preceq 0$ is a constant matrix, independent of $\bm{x}_i$. Subsequently, the ROM is a linear discrete model. It follows that the model predicted equilibrium state is equal to the initial values $\bm{x}(0)$ projected onto the kernel of $T$, which agrees with the data. 

\begin{figure}
    \centering
    \includegraphics[width=\columnwidth]{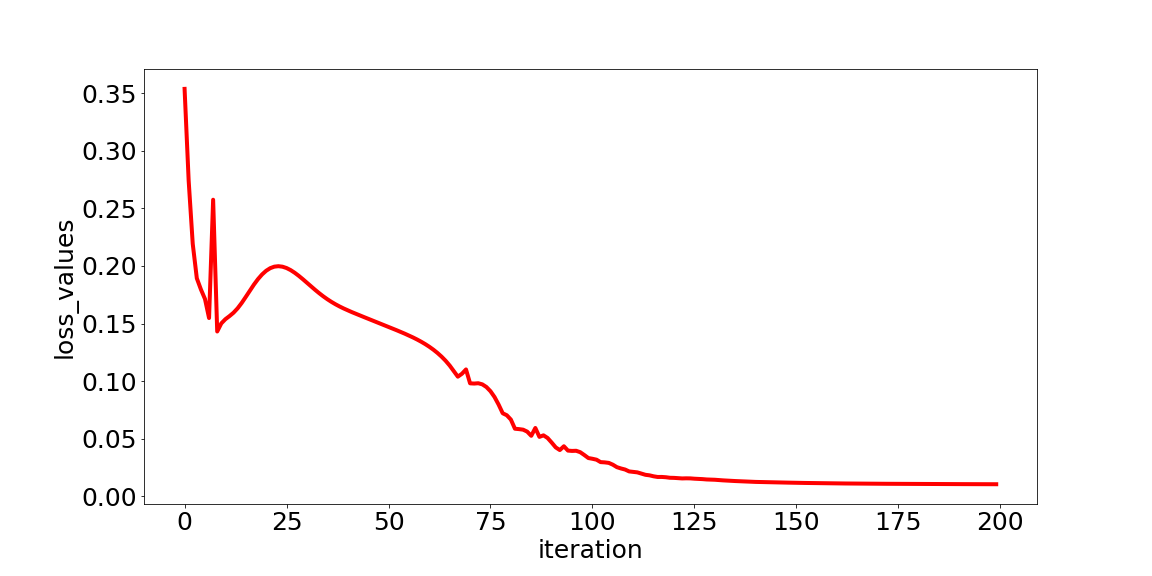}
    \caption{Example \ref{sec: ex_rub_l}: {Rubber band example with $R_L=0$.} Training loss v.s. \# of iterations. Weights are initialized by the Glorot initializer. ADAM optimizer is used with learning rate $10^{-3}$.}
    \label{fig: rbl_loss}
\end{figure}

\begin{figure}
    \centering
    \includegraphics[width=.9\columnwidth]{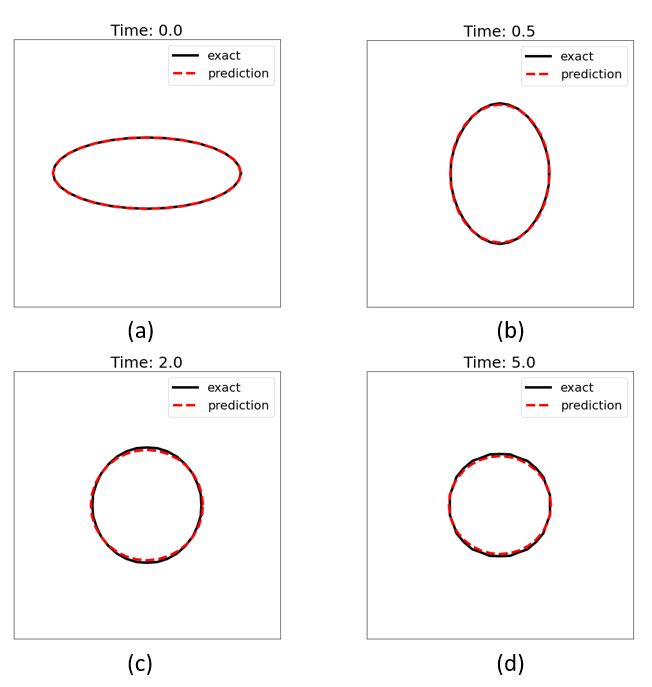}
    \caption{Example \ref{sec: ex_rub_l}: {Rubber band oscillation with $R_L=0$.} Snapshots of the rubber band at different time steps. (a), (b): training fit. (c), (d): comparisons of the testing phase.}
    \label{fig: rbl_snaps}
\end{figure}

\begin{figure}
    \centering
    \includegraphics[width=\columnwidth]{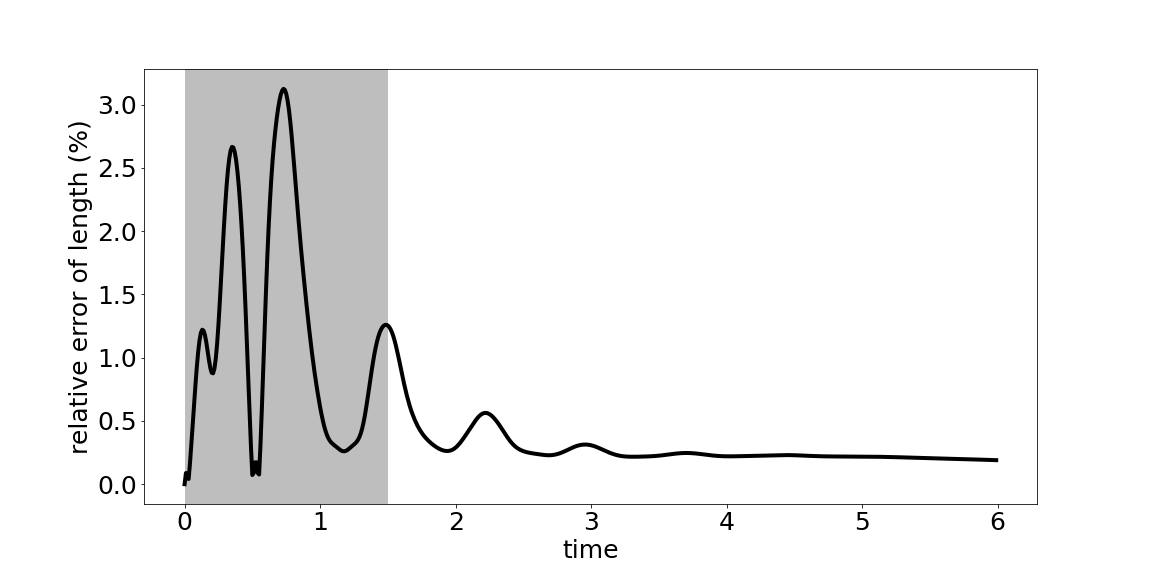}
    \caption{Example \ref{sec: ex_rub_l}: {Rubber band oscillation with $R_L=0$.} Relative error in perimeters against time steps. The shaded area covers the first 150 steps.}
    \label{fig: rbl_pm}
\end{figure}

\subsubsection{Nonlinear force ($R_L\neq0$)}\label{sec: ex_rub_nl}
We set $R_L$ to be the smallest one among all virtual springs connecting the structure points at the beginning, so that the rubber band is initially stretched everywhere. The training converges in less than 200 iterations (see Fig.~\ref{fig: rbnl_loss}). Configurations of the rubber band at different time steps are compared in Fig.~\ref{fig: rbnl_snaps}. The relative error in the rubber band perimeter never exceeds 1\% (see Fig.~\ref{fig: rbnl_pm}).

\begin{figure}
    \centering
    \includegraphics[width=\columnwidth]{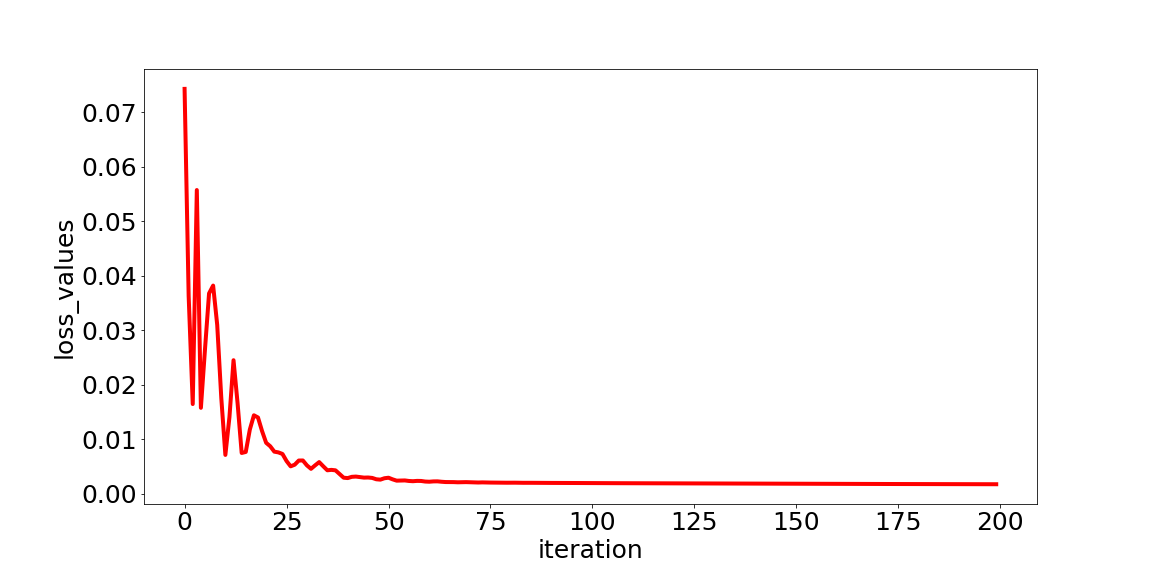}
    \caption{Example \ref{sec: ex_rub_nl}: {Rubber band oscillation with $R_L\neq0$.} Training loss of the case $R_L\neq 0$. Weights are initialized by the Glorot initializer. ADAM optimizer is used with learning rate $10^{-3}$.}
    \label{fig: rbnl_loss}
\end{figure}

\begin{figure}
    \centering
    \includegraphics[width=.9\columnwidth]{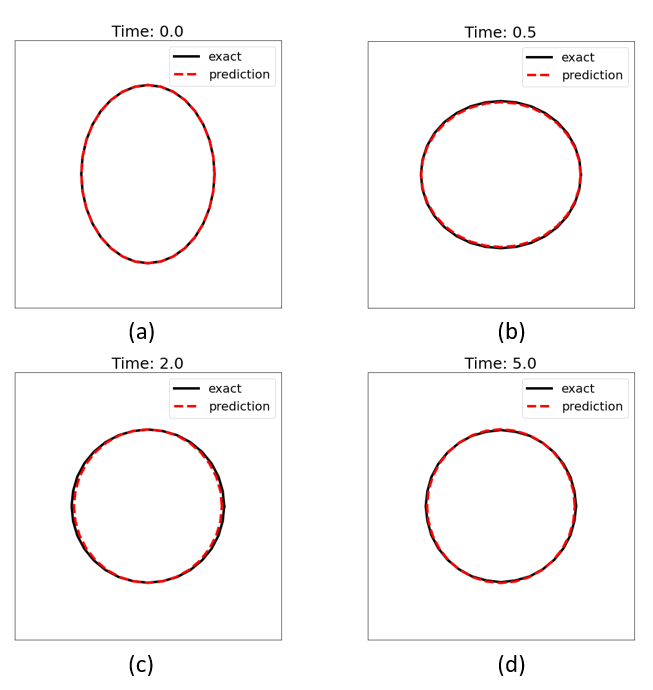}
    \caption{Example \ref{sec: ex_rub_nl}: {Oscillating rubber band with $R_L\neq0$.} Snapshots of the rubber band at different time steps. (a), (b): training fit. (c), (d): comparisons of the testing phase.}
    \label{fig: rbnl_snaps}
\end{figure}

\begin{figure}
    \centering
    \includegraphics[width=\columnwidth]{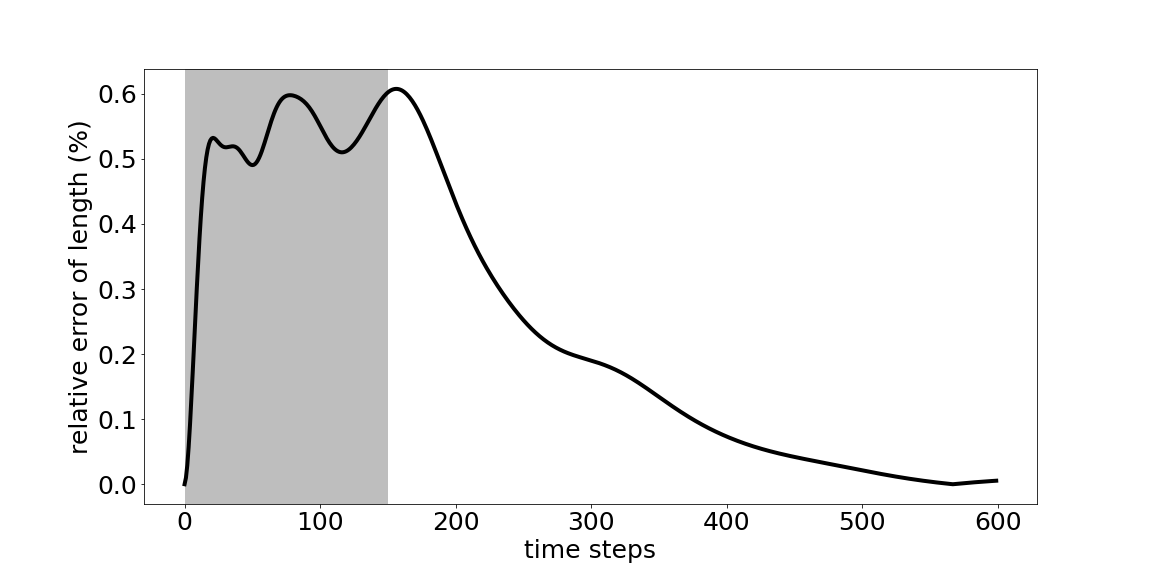}
    \caption{Example \ref{sec: ex_rub_nl}: {Oscillating rubber band with $R_L\neq0$.} Relative error in perimeters against time steps. The shaded area covers the first 150 steps.}
    \label{fig: rbnl_pm}
\end{figure}

\subsection{Ellipsoid rotation}\label{sec: ex_rot}
In the rubber band example, models with stationary equilibrium states are learned from data. Our reduced-order modeling technique is also able to learn ROMs for systems exhibiting periodic motions. In this example, we simulate the rotation of a rigid ellipsoid in a shear flow, which has been studied as a benchmark example \cite{hao2015fictitious} in FSI. Fig.~\ref{fig: orbdem} illustrates the problem setup, where $a$ and $b$ are the semi-major and semi-minor axis, respectively; $r$ is the shear rate of the background flow. In Section \ref{sec: ex_roto}, we study the case of $u_0=0$ such that the ellipsoid rotates without translation. In Section \ref{sec: ex_rott}, $u_0$ is a nonzero constant for all $y$ so that the ellipsoid is rotating with its center moving horizontally. The semi-major axis is initially along the $y$-axis.

\begin{figure}
    \centering
    \includegraphics[width=\columnwidth]{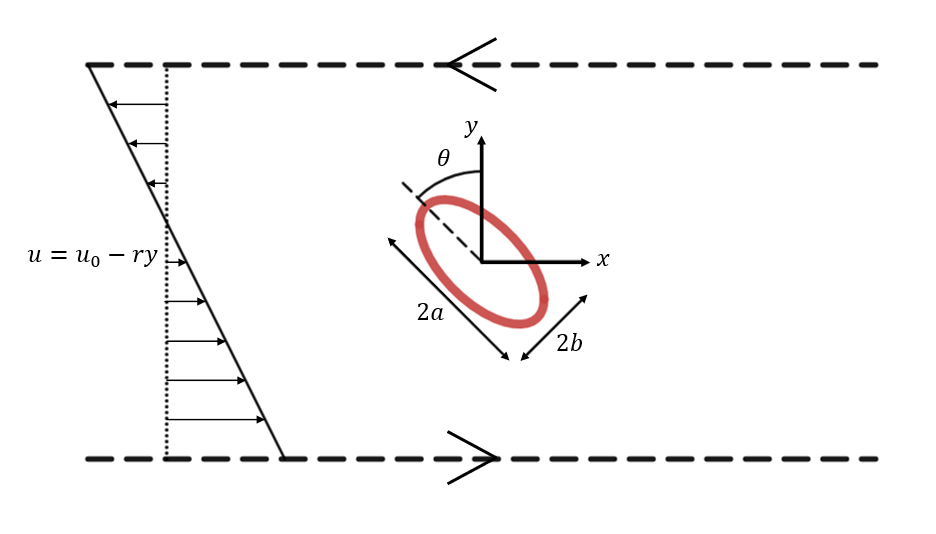}
    \caption{Example \ref{sec: ex_rot}: A rigid ellipse immersed in a shear flow.}
    \label{fig: orbdem}
\end{figure}

It has been shown that the instantaneous inclination angle $\theta$ of the major axis with respect to the $y$-axis is
\begin{equation}\label{eq: orbittrue}
    \tan(\theta) = \frac{a}{b}\tan\left(\frac{ab}{a^2+b^2}rt\right),
\end{equation}
where $t$ is the time variable \cite{jeffery1922motion}. 

To preserve the elliptic shape of the rigid structure, the body force in this example is from a discrete bending energy \cite{pivkin2008accurate},
\begin{equation}
E_b = \sigma_b \sum_{i=1}^{n_s}(1-\cos(\omega_i-\omega_i^0)),
\end{equation}
where $\omega_i^0$ is the initial angle between the adjacent edges with the $i$-th Lagrangian grid point, $\omega_i$ is the current angle, $n_s$ is the number of Lagrangian grid points and $\sigma_b$ is the bending coefficient. The bending force generated on each structure point is given by, 
\begin{equation}
f_i = ({f_i} _x,{f_i} _y) = (-\frac{\partial E_b}{\partial x_i},-\frac{\partial E_b}{\partial y_i} ).
\end{equation}

We sample 32 points on the ellipsoid ($l=64$). Data is computed directly from the analytical solution \eqref{eq: orbittrue}. The data set contains one time series sampled with $\Delta t=10^{-2}$. A model with latent dimension $p=20$ is enough to capture the periodic motion. The first 3 quarters of a rotational period (150 steps, specifically) are used for training. The trained ROM is tested against the exact data on the following time steps till the end of the second period.

\subsubsection{Rotation only ($u_0=0$)}\label{sec: ex_roto}
Snapshots of the ellipsoid are plotted to show the fitting quality and prediction accuracy (see Fig.~\ref{fig: roto_snaps}. We also compare the rotation angle given by the ROM with the analytical solution (see Fig.~\eqref{eq: orbittrue}). The ROM successfully learns and predicts the periodic dynamics. Since we know that $\Delta t=10^{-2}$, we could obtain the continuous model corresponding to the discrete ROM. The eigenvalues of the approximate continuous model is shown in Fig.~\ref{fig: roto_eigs}, with real parts in the range $[-1.56\times 10^{-4}, 0]$. Therefore, the system exhibits mostly periodic motion with very slow decay.

\begin{figure}
    \centering
    \includegraphics[width=.9\columnwidth]{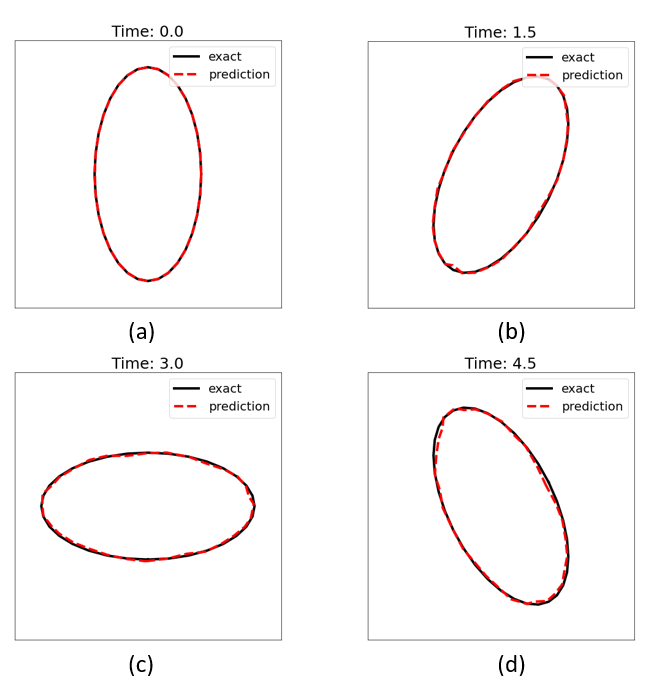}
    \caption{Example \ref{sec: ex_roto}: The dynamics of a rotating ellipsoid. Shapes of the ellipsoid at different time steps.
    (a), (b): training fit. (c), (d): testing phase.}
    \label{fig: roto_snaps}
\end{figure}

\begin{figure}
    \centering
    \includegraphics[width=\columnwidth]{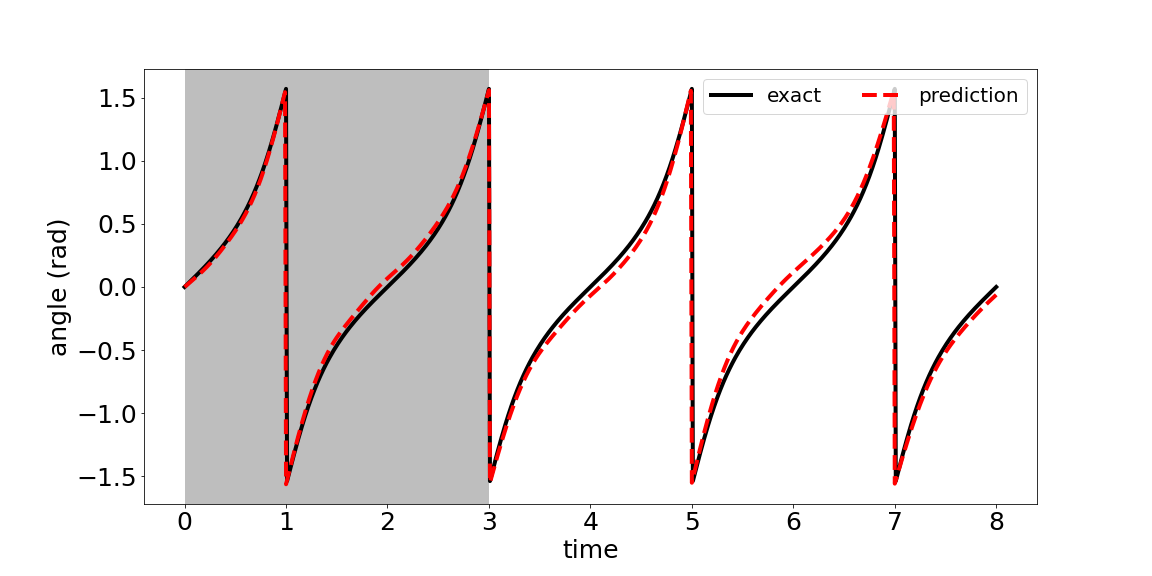}
    \caption{Example \ref{sec: ex_roto}: The dynamics of a rotating ellipsoid. Plot of inclination angles against time steps. Shaded area indicates the 150 time steps used for training.}
    \label{fig: roto_ang}
\end{figure}

\begin{figure}
    \centering
    \includegraphics[width=\columnwidth]{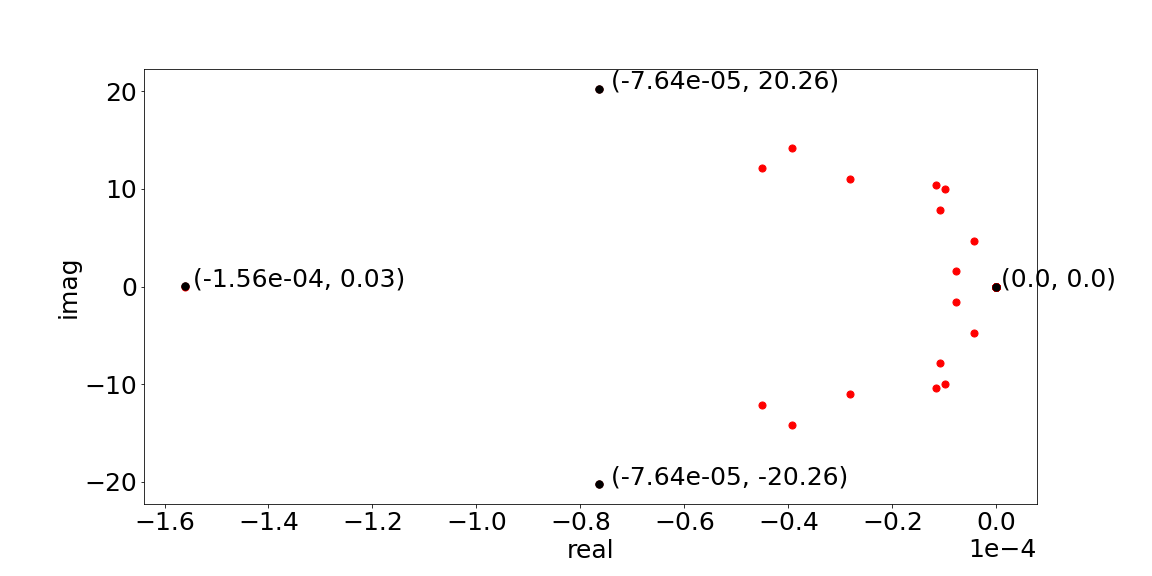}
    \caption{Example \ref{sec: ex_roto}: The dynamics of a rotating ellipsoid. Eigenvalues of the approximate continuous model corresponding to the trained ROM.}
    \label{fig: roto_eigs}
\end{figure}

\subsubsection{Rotation and translation ($u_0\neq 0$)}\label{sec: ex_rott}
In this case, we add another trainable vector parameter $\bm{z}_0$ to \eqref{eq: out_up} so that 
\begin{equation}\label{eq: out_up_c}
    \bm{x}_{i+1} = \bm{x}_i + \frac{R_d^T}{2}(\bm{z}_{i+1}+\bm{z}_i+2\bm{z}_0).
\end{equation}
The added parameter $\bm{z}_0$ is intended to capture the the effect of $u_0$. Since the latent variable $\bm{z}$ is acting as the virtual fluid variables around the structure, $\bm{z}_i-\bm{z}_0$ in this case follows closely the dynamics of $\bm{z}_i$ in the non translational case.

As before, we plot the snapshots of the ellipsoid (Fig.~\ref{fig: rott_snaps}) and the rotation angles (Fig.~\ref{fig: rott_ang}) to show the goodness of fit. In addition, we compare the x-coordinates of the center given by the ROM with the exact data (see Fig.~\ref{fig: rott_disp}). These numerical results show that the ROM \eqref{eq: state_up}, with \eqref{eq: out_up} replaced by \eqref{eq: out_up_c}, learns both the rotational and translational motion from data. 
\begin{figure}
    \centering
    \includegraphics[width=\columnwidth]{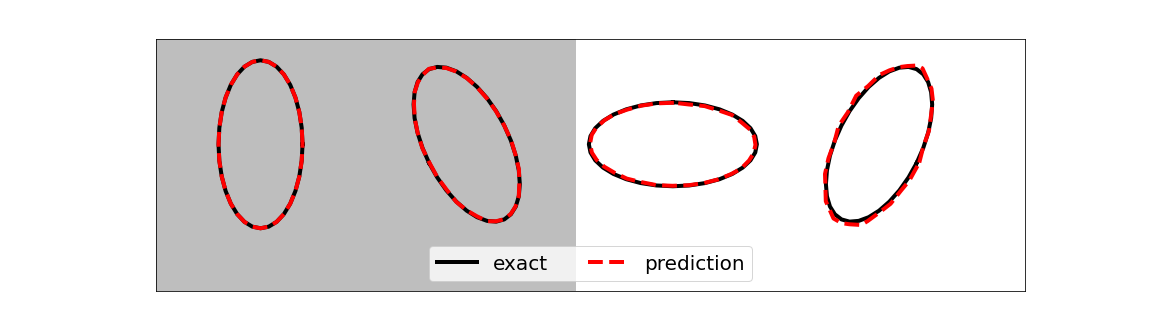}
    \caption{Example \ref{sec: ex_rott}: {Ellipsoid translational rotation.} Shapes of the ellipsoid at different time steps. Shaded area covers the interval used for training.}
    \label{fig: rott_snaps}
\end{figure}

\begin{figure}
    \centering
    \includegraphics[width=\columnwidth]{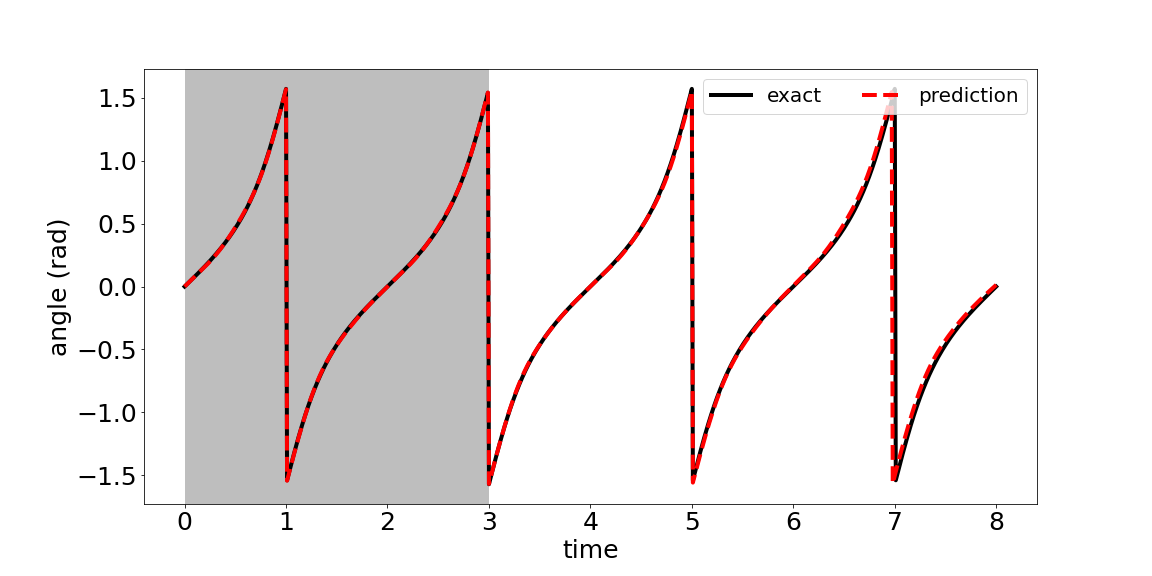}
    \caption{Example \ref{sec: ex_rott}: {Ellipsoid translational rotation.} Plot of inclination angles against time. Shaded area covers the time steps used for training.}
    \label{fig: rott_ang}
\end{figure}

\begin{figure}
    \centering
    \includegraphics[width=\columnwidth]{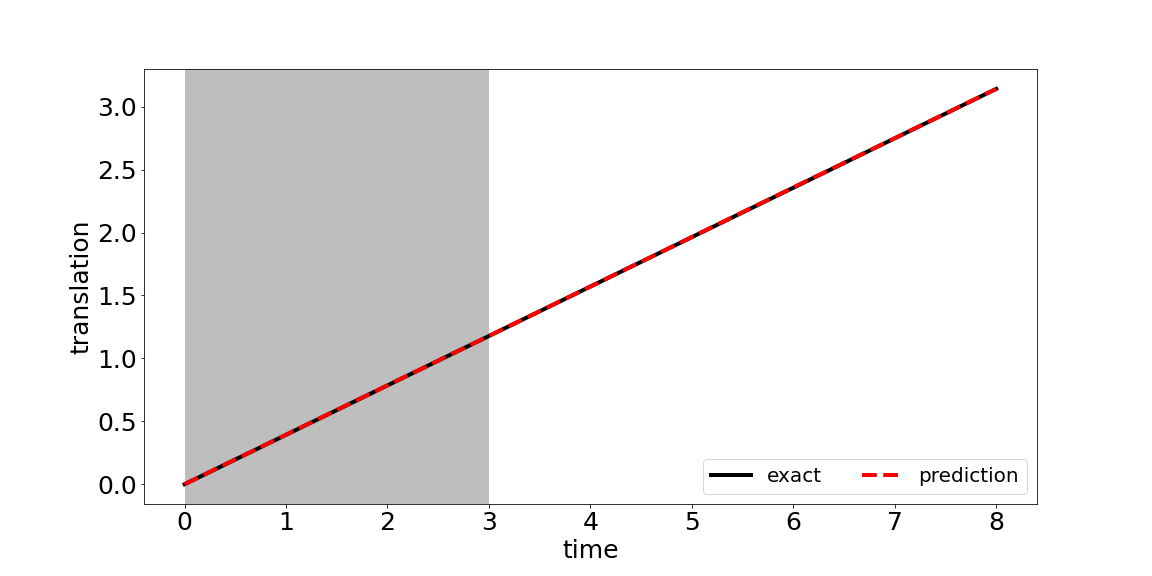}
    \caption{Example \ref{sec: ex_rott}: {Ellipsoid with both translation and rotation.} Displacement of ellipsoid center against time. Shaded area covers the time steps used for training.}
    \label{fig: rott_disp}
\end{figure}

\section{Discussion}\label{sec: conclusions}

This paper presents a data-driven modeling framework to incorporate the dynamics of latent variables. The parametric form is motivated by reduced-order modeling as well as fluid structure interaction problems.  
The proposed modeling technique provides a general framework for data-driven modeling requiring stability and coupling with unobserved quantities.   
In particular, enforcing the stability is straightforward.  Recent developments in deep learning provide an efficient constrained optimization framework to train the discrete models using recurrent cells and back propagation through time. Numerical results show that the stability conditions are essential. Using examples from fluid-structure interactions, our numerical tests indicate that relatively small number of the latent variables are already sufficient to capture the dynamics. 

It is worthwhile to point out that we tune the dimension of the latent variable following bottom-up constructions \cite{freund1999reduced,Bai2002}, i.e., they are multiples of the dimension of observed variables. However, in the last example, we observe that it is possible to have a good model, where the dimension of the latent variables is smaller  than that of the observed variable.

To increase the flexibility of our model, a simple idea is to include non-constant coupling  between the observed variable and the latent variable. Namely, one can make a variable matrix $L$ that is $\bm{x}$-dependent. Such extension does not affect the Lyapunov stability of the continuous model, as theorem \ref{thm: fullfsi_stab} would still hold. Therefore, by Theorem \ref{thm: mid_quad}, the implicit mid-point scheme inherits stability. However, since the forward propagation involves solving nonlinear equations, the back propagation of such a discrete model is not as straightforward. It may require modifications of existing software packages, e.g., Tensorflow.  
Another potential extension is to introduce nonlinearity  in the interactions of latent variables,  e.g., by neural networks. Linear stability could be achieved, provided that the nonlinear part is of higher order, e.g., see \cite[Theorem 7.1]{verhulst2006nonlinear}.

\section*{Data Availability}
The code and data that support the findings of this study are openly available on Github \cite{luo2022rnn}.

\bibliographystyle{ieeetr}
\bibliography{rom_rnn,addref}

\end{document}